\documentclass[10pt]{article}
\usepackage{amssymb,latexsym, amsmath, enumerate, amsthm, mathscinet, mathtools}
\usepackage{cases, verbatim}
\usepackage[active]{srcltx}
\usepackage{hyperref}
\hypersetup{
      CJKbookmarks=true
	colorlinks=true,
	linkcolor=blue,
	filecolor=blue,      
	urlcolor=blue,
	citecolor=blue,
}
\usepackage{xcolor}

\def\R{\mathbb R}

\def\C{\mathbf C}
\def\Q{\mathbf Q}

\def\b{\bold}
\def\I{\mathrm{id}}
\def\l{\langle}
\def\r{\rangle}
\def\tr{\mathrm{Tr}}
\def\SS{\mathcal{S}^+}
\def\SSS{\mathcal{S}}
\def\D{\mathbf D}

\newtheorem{theorem}{Theorem}
\newtheorem{lemma}{Lemma}
\newtheorem{remark}{Remark}
\newtheorem{corollary}{Corollary}
\newtheorem{proposition}{Proposition}

\title{On the heat semigroup approach to the geometric Forward-Reverse Brascamp--Lieb inequality}
\author{Ye Zhang}
\date{}

\begin{document}

\renewcommand{\theequation}{\thesection.\arabic{equation}}
 \maketitle

\vspace{-1.0cm}

\bigskip
	
	{\bf Abstract.} In this short paper we provide a new proof of the geometric Forward-Reverse Brascamp--Lieb inequality, using the approach of the heat semigroup, or the heat flow. Furthermore, we characterize all the Forward-Reverse Brascamp--Lieb data such that the initial relation can be preserved by some heat flow.
	
	\medskip
	
	
	\medskip
	

\bigskip

\section{Introduction}
\setcounter{equation}{0}

Recall that the Forward-Reverse Brascamp--Lieb inequality, studied recently in \cite{LCCV18, CL21}, unifies both the Brascamp--Lieb inequality and Barthe's inequality (which is also called reverse Brascamp--Lieb inequality). As a result, it generalizes several famous inequalities, such as H\"older's inequality, Young's inequality, Loomis--Whitney inequality, Pr\'ekopa--Leindler inequality, etc.

\subsection{Notation}


Throughout this paper, by {\it Euclidean space $V$}, we mean a finite dimensional Hilbert space equipped with the inner product $\l \cdot, \cdot \r_V$ and the Lebesgue measure (more precisely, the measure induced by the Riemannian volume form if we regard $V$ as a Riemannian manifold). On the direct sum of two Euclidean spaces $V_1 \oplus V_2$, the inner product is defined by
\[
\l v_1 + v_2, w_1 + w_2 \r_{V_1 \oplus V_2} := \l v_1, w_1 \r_{V_1} + \l v_2, w_2 \r_{V_2}, \qquad \forall \, v_1, w_1 \in V_1, v_2, w_2 \in V_2.
\]
With this inner product on $V_1 \oplus V_2$ the subspaces $V_1$ and $V_2$ are orthogonal with each other. We use $|\cdot|_V$ to denote the norm induced by the inner product $\l \cdot, \cdot \r_V$. Given a Euclidean space $V$ and a linear map $A$ from $V$ to itself, we use $\tr(A)$ and $\det(A)$ to denote its {\it trace} and {\it determinant} respectively. Moreover, for a linear map between two  Euclidean spaces $A: V_1 \to V_2$, we denote the {\it adjoint of $A$} by $A^*$, that is, $A^*$ is a linear map from $V_2$ to $V_1$ satisfying
\[
\l A v_1, v_2 \r_{V_2} = \l v_1, A^* v_2 \r_{V_1}, \qquad \forall \, v_1 \in V_1, v_2 \in V_2.
\]
We use $\SSS(V)$ to denote the set of self-adjoint linear maps on $V$, namely set of linear map $A$ from $V$ to itself such that $A = A^*$. The notation $\SS(V)$ (resp. $\overline{\SS}(V)$) is reserved for the set of self-adjoint positive  definite (resp. semidefinite) linear maps on $V$, that is, $A \in \SS(V)$ (resp. $A \in \overline{\SS}(V)$) if and only if
\[
A \in \SSS(V) \qquad \mbox{and} \qquad \l A v,v \r_V > 0 \ \mbox{(resp. $ \l A v,v \r_V \ge 0$)}, \qquad \forall \, v  \in V \setminus \{0\}.
\]
For $A \in \overline{\SS}(V)$, we use $\sqrt{A}$ to denote its {\it square root}, namely the unique element $B \in \overline{\SS}(V)$ such that $B^2 = A$. For two elements $A, B \in \SSS(V)$, we say $A \ge B$ if $A - B \in \overline{\SS}(V)$. A {\it centred Gaussian function} on a Euclidean space $V$ is a function with the following form:
\[
x \mapsto \exp\left( -\langle A x,x \rangle_V \right), \qquad \forall \, x \in V, A \in \SS(V).
\]

Furthermore, for a given Euclidean space $V$, we denote the {\it Laplacian} and {\it gradient} by $\Delta_V$ and $\nabla_V$ respectively. The {\it Hessian matrix} of a $C^2$ function $f$ is denoted by $\nabla^2_V f$. Recall that $\Delta_V f = \tr(\nabla_V^2 f)$. Also recall that $\Delta_V = \mathrm{div}_V (\nabla_V)$, where $\mathrm{div}_V$ is the {\it divergence operator}. More generally, for $A \in \SS(V)$,  we use $\Delta_{V,A}$ to denote the following operator:
\[
\Delta_{V,A} f:= \mathrm{div}_V (A \nabla_V f), \qquad \forall \, f \in C^2.
\]
In the following we use $(e^{t \Delta_{V, A}})_{t > 0}$ to denote the {\it heat semigroup} with infinitesimal generator $\Delta_{V, A}$ on $L^1(V)$. In fact, for $f \in L^1(V)$, we have
\begin{align}\label{conv}
(e^{t \Delta_{V, A}} f)(x) = \int_V f(y) p_t^{V,A}(x - y) dy = \int_V f(x - y) p_t^{V,A}(y) dy,
\end{align}
where the {\it heat kernel} of $(e^{t \Delta_{V, A}})_{t > 0}$ is given by the following explicit formula
\begin{align}\label{exphk}
p_t^{V,A}(x) = \frac{1}{(4\pi t)^{\frac{\mathrm{dim} V }{2}} \sqrt{\det(A)}} e^{-\frac{\langle A^{-1} x,x \rangle_V}{4t}}, \qquad x \in V.
\end{align}
When $A = \I_V$, {\it the identity matrix} on $V$, we just write $(e^{t \Delta_{V}})_{t > 0}$ and $p_t^{V}$ for simplicity. From the expression above, we know that the function $u(x,t) := (e^{t \Delta_{V, A}} f)(x)$ is a  smooth function on $V \times (0,+\infty)$ if $f \in L^1(V)$. If in addition we have $f \ge 0$ with $\int_V f > 0$, the function $u$ is also  positive. Note that by formula \eqref{conv}, the heat semigroup can also be defined on the space of non-negative, measurable functions. For example, we have $e^{t \Delta_{V, A}} 1 = 1$. However, in this case, $e^{t \Delta_{V, A}} f$ may take value $+\infty$ at some point for general $f$. For more details of the heat semigroup, we refer to \cite[\S~2.7]{G09} (see also \cite[Theorem 7.19]{G09} for the smoothness result).

Given Euclidean spaces $(E_i)_{1 \le i \le k}$ and $(E^j)_{1 \le j \le m}$, we define $E_0 := \oplus_{i = 1}^k E_i$, and $E^0 := \oplus_{j = 1}^m E^j$. We use $\pi_{E_i}$ and $\pi_{E^j}$ to denote the orthogonal projection of $E_0$ onto $E_i$ and  $E^0$ onto $E^j$ respectively. In this paper we always regard elements in Euclidean spaces as column vectors. In this spirit throughout this paper we identify the spaces $E_0 = \oplus_{i = 1}^k E_i$ and $E^0 = \oplus_{j = 1}^m E^j$ with $\prod_{i = 1}^k E_i$ and $\prod_{j = 1}^m E^j$
with the following maps:
\begin{align}\label{iden}
\sum_{i = 1}^k v_i \mapsto \begin{pmatrix}
v_1 \\
v_2 \\
\vdots \\
v_k
\end{pmatrix}, 
\qquad
\sum_{j = 1}^m v^j \mapsto \begin{pmatrix}
v^1 \\
v^2 \\
\vdots \\
v^m
\end{pmatrix},
\qquad \forall \, v_i \in E_i, v^j \in E^j, 1 \le i \le k, 1 \le j \le m.
\end{align}

Let $\Q$ be a linear map from $E_0$ to $E^0$, and  $\b{c} := (c_i)_{1 \le i \le k}, \b{d} := (d_j)_{1 \le j \le m}$ be collections of positive numbers such that
\begin{align}\label{ass0}
\sum_{i = 1}^k c_i \mathrm{dim}(E_i) = \sum_{j = 1}^m d_j \mathrm{dim}(E^j).
\end{align}
Such triple $(\b{c}, \b{d}, \Q)$ is called a {\it  Forward-Reverse Brascamp--Lieb datum}.

For $A_i \in \overline{\SS}(E_i), i = 1, \ldots, k$ (resp. $A^j \in \overline{\SS}(E^j), j = 1, \ldots, m$), we let
$\Pi(A_1, \ldots, A_k)$ (resp. $\Pi(A^1, \ldots, A^m)$) denote the set of $A_0 \in \overline{\SS}(E_0)$ (resp.$A^0 \in \overline{\SS}(E^0)$) satisfying
\[
\pi_{E_i} A_0 \pi_{E_i}^* = A_i, \quad \forall \, 1 \le i \le k \qquad \mbox{(resp. $\pi_{E^j} A^0 \pi_{E^j}^* = A^j, \quad \forall \, 1 \le j \le m$).}
\]
In other words, under the identification \eqref{iden}, $A_0$ and $A^0$ have the following block forms:
\[
A_0 = \begin{pmatrix}
A_1 & * & \ldots & * \\
* & A_2 & \ldots & * \\
\vdots & \vdots & \ddots & \vdots \\
* & * & \ldots & A_k
\end{pmatrix},
\qquad 
A^0 = \begin{pmatrix}
A^1 & * & \ldots & * \\
* & A^2 & \ldots & * \\
\vdots & \vdots & \ddots & \vdots \\
* & * & \ldots & A^m
\end{pmatrix}.
\]

Finally for $\b{c} := (c_i)_{1 \le i \le k}, \b{d} := (d_j)_{1 \le j \le m}$, we define $\Lambda_{\b{c}} := \sum_{i =1}^k c_i \pi_{E_i}^* \I_{E_i} \pi_{E_i}$ and $\Lambda_{\b{d}} := \sum_{j =1}^m d_j\pi_{E^j}^* \I_{E^j} \pi_{E^j}$. Note that with the identification \eqref{iden}, $\Lambda_{\b{c}}$ and $\Lambda_{\b{d}}$ have the following block forms:
\[
\Lambda_{\b{c}} = \begin{pmatrix}
c_1 \I_{E_1} & 0 & \ldots & 0 \\
0 & c_2 \I_{E_2} & \ldots & 0 \\
\vdots & \vdots & \ddots & \vdots \\
0 & 0 & \ldots & c_k \I_{E_k}
\end{pmatrix},
\quad 
\Lambda_{\b{d}} = \begin{pmatrix}
d_1\I_{E^1} & 0 & \ldots & 0 \\
0 & d_2\I_{E^2} & \ldots & 0 \\
\vdots & \vdots & \ddots & \vdots \\
0 & 0 & \ldots & d_m\I_{E^m}
\end{pmatrix}.
\]

\subsection{Forward-Reverse Brascamp--Lieb inequality}

In our setting, given a Forward-Reverse Brascamp--Lieb datum $(\b{c}, \b{d}, \Q)$, the {\it Forward-Reverse Brascamp--Lieb inequality}, is the following:
\begin{align}\label{FRBL}
\prod_{i = 1}^k \left( \int_{E_i} f_i \right)^{c_i} \le C(\b{c}, \b{d}, \Q) \prod_{j = 1}^m  \left( \int_{E^j} g_j \right)^{d_j}
\end{align}
for all non-negative measurable functions $(f_i)_{1 \le i \le k}, (g_j)_{1 \le j \le m}$ satisfying
\begin{align}\label{ass3}
\prod_{i = 1}^k f_i^{c_i}(\pi_{E_i} x) \le \prod_{j = 1}^m g_j^{d_j}(\pi_{E^j} \Q x), \qquad \forall \, x \in E_0.
\end{align}
Here the {\it Forward-Reverse Brascamp--Lieb constant} $C(\b{c}, \b{d}, \Q)$ is the smallest constant for which \eqref{FRBL} holds and it could be $+\infty$. Measurable functions $(f_i)_{1 \le i \le k}, (g_j)_{1 \le j \le m}$ are called {\it extremizers} if \eqref{FRBL} holds with equality (with positive, finite value on both sides). Moreover, if $(f_i)_{1 \le i \le k}, (g_j)_{1 \le j \le m}$ are extremizers and all the functions $(f_i)_{1 \le i \le k}, (g_j)_{1 \le j \le m}$ are centred Gaussian functions, we call them {\it Gaussian extremizers}.

\begin{remark}
Actually \eqref{ass0} is a necessary condition to make $C(\b{c}, \b{d}, \Q) < +\infty$, which can be seen by a scaling argument. 
\end{remark}

\begin{remark}\label{notation}
Here we follow the notation in \cite{LCCV18}, which is a little bit different from the one in \cite{CL21}. As mentioned in \cite[Remark 1.5]{CL21}, the extra scaling in \cite{CL21} is due to the duality relation there. Since we will not consider the duality relation here, we can absorb the scaling into the map. Furthermore, if we write 
\begin{align}\label{defQji}
\Q_{ji} := \pi_{E^j} \Q \pi^*_{E_i}, \qquad \forall \, 1 \le i \le k, 1 \le j \le m,
\end{align}
which is a linear map from $E_i$ to $E^j$, we have $\Q_{ji} = c_i B_{ij}$ in \cite[(2)]{CL21}. Note that with the identification \eqref{iden}, we can write
\begin{align}\label{maQ}
\Q = \begin{pmatrix}
\Q_{11} & \Q_{12} & \ldots & \Q_{1k} \\
\Q_{21} & \Q_{22} & \ldots & \Q_{2k} \\
\vdots & \vdots & \ddots & \vdots \\
\Q_{m1} & \Q_{m2} & \ldots & \Q_{mk}
\end{pmatrix}
= \begin{pmatrix}
\Q_1 \\
\Q_2 \\
\vdots \\
\Q_m
\end{pmatrix},
\end{align}
where
\begin{align}\label{defQj}
\Q_j := \sum_{i = 1}^k \Q_{ji} \pi_{E_i} = \pi_{E^j} \Q, \qquad \forall \, 1 \le j \le m.
\end{align}
\end{remark}

\begin{remark}
If $k = 1$ and $\b{c} = (1)$, the Forward-Reverse Brascamp--Lieb inequality reduces to the (classical) Brascamp--Lieb inequality. On the other hand, if $m = 1$ and $\b{d} = (1)$, it reduces to the Barthe's inequality, or reverse Brascamp--Lieb inequality. In such cases we simply say Brascamp--Lieb constant (resp. datum) and Reverse Brascamp--Lieb constant (resp. datum) instead of Forward-Reverse Brascamp--Lieb constant (resp. datum).
\end{remark}

The (classical) Brascamp--Lieb inequality was first formulated in \cite{BL76} to find the best constants in Young's inequality. By  Lieb's Theorem (see \cite[Theorem 6.2]{L90}), the Brascamp--Lieb constant can be obtained by only considering centred Gaussian functions. In other words, the Brascamp--Lieb constant is exhausted by centred Gaussian functions. Then in \cite{BCCT08} the authors gave a sufficient and necessary condition on the finiteness of the Brascamp--Lieb constant, as well as on the existence of the extremizers. Finally  all the extremizers are characterized  in \cite{V08}.

On the other hand, the Reverse Brascamp--Lieb inequality was formulated in \cite{B982}, where a counterpart of Lieb's Theorem is also proved. However, the extremizers are only characterized in special cases in \cite{BKX23}.

For the case of the Forward-Reverse Brascamp--Lieb inequality, a generalization of Lieb's Theorem is proved in both \cite{LCCV18, CL21}. A sufficient and necessary condition on the finiteness of the  Forward-Reverse Brascamp--Lieb  constant is also provided in \cite{CL21}.

\subsection{Geometric Forward-Reverse Brascamp--Lieb datum}\label{ss13}

By a clever induction argument (cf. \cite{CL21}, see also \cite{BCCT08} for the case of Brascamp--Lieb inequality), it turns out that the study of general Forward-Reverse Brascamp--Lieb inequalities essentially  reduced to the geometric ones. A Forward-Reverse Brascamp--Lieb datum is called  {\it geometric} if 
\begin{align}\label{ass2}
\Q^* \Lambda_{\b{d}} \Q \le \Lambda_{\b{c}}
\end{align}
and there exists a $\Sigma \in \Pi(\I_{E_1}, \ldots, \I_{E_k})$ such that $\Q \Sigma \Q^* \in \Pi(\I_{E^1}, \ldots, \I_{E^m})$. By Remark \ref{notation}, our definition here is the same as the one in \cite[Definition 1.21]{CL21} but we write in a more compact way by using matrices. This will help us do the computation in the proof of our main theorem.

In the case $k = 1$ and $\b{c} = (1)$, the only choice of $\Sigma$ is $\I_{E_1}( = \I_{E_0})$. As a result, the geometric conditions reduce to
\[
\Q_j \Q_j^* = \I_{E^j}, \quad \forall \, 1 \le j \le m, \qquad \sum_{j = 1}^m d_j \Q_j^* \Q_j \le \I_{E_0},
\]
where $\Q_j$ is defined by \eqref{defQj}. By \eqref{ass0}, the traces of the left-hand side and right-hand side of the last inequality coincide, which implies equality, that is, 
\[
\sum_{j = 1}^m d_j \Q_j^* \Q_j = \I_{E_0}.
\]
Thus, the geometric conditions here reduce to the ones of the geometric Brascamp--Lieb inequality (cf. \cite[Definition 2.1]{BCCT08}).

Given two Forward-Reverse Brascamp--Lieb data $(\b{c}, \b{d}, \Q)$ and $(\b{c}', \b{d}', \Q')$, we say they are {\it equivalent} if $\b{c} = \b{c}', \b{d} = \b{d}'$, and there exist invertible linear maps $C_i : E_i \to E_i, 1 \le i \le k$ and  $D_j : E^j \to E^j, 1 \le j \le m$ such that
\[
\Q_{ji}' = D_j^{-1} \Q_{ji} C_i^{-1}, \qquad \forall \, 1 \le i \le k, 1 \le j \le m,
\]
where $\Q_{ji}$ is defined in \eqref{defQji}.

\subsection{Main result}

In the literature, there are several methods of proofs for the Brascamp--Lieb inequality and Reverse Brascamp--Lieb inequality: (1) by rearrangement inequality \cite{BL76, BLL74}; (2) by optimal transport \cite{B982, B98}; (3) by a probabilistic argument \cite{B00, L14}; (4) by heat semigroup \cite{BC04, BCLM11,  BH09, BCCT08, B00, CC09, CLL04}. Regarding the Forward-Reverse Brascamp--Lieb inequality, the proof given in \cite{CL21} is by a probabilistic argument developing the ones in \cite{B00, L14}. The target of this paper is to give a new proof of the geometric Forward-Reverse Brascamp--Lieb inequality by using the heat semigroup. Here is our main result: if the Forward-Reverse Brascamp--Lieb datum $(\b{c}, \b{d}, \Q)$ is geometric, then the relation \eqref{ass3} can be preserved by the heat flow.

\begin{theorem}\label{mthm}
Given a geometric Forward-Reverse Brascamp--Lieb datum $(\b{c}, \b{d}, \Q)$, if all the functions $(f_i)_{1 \le i \le k}, (g_j)_{1 \le j \le m}$ are non-negative and integrable, and \eqref{ass3} holds, namely
\[
\prod_{i = 1}^k f_i^{c_i}(\pi_{E_i} x) \le \prod_{j = 1}^m g_j^{d_j}(\pi_{E^j} \Q x), \qquad \forall \, x \in E_0,
\]
then we have
\begin{align}\label{heat}
\prod_{i = 1}^k (e^{t\Delta_{E_i}}f_i)^{c_i}(\pi_{E_i} x) \le \prod_{j = 1}^m (e^{t\Delta_{E^j}}g_j)^{d_j}(\pi_{E^j} \Q x), \qquad \forall \, x \in E_0, t > 0.
\end{align}
\end{theorem}

\begin{remark}
Recall that for the geometric Brascamp--Lieb inequality, the heat semigroup approach is a little bit different from what we stated above. To be more precise, given a geometric Forward-Reverse Brascamp--Lieb datum with $k = 1$ and $\b{c} = (1)$ (see Subsection \ref{ss13} for details), instead of the preservation property under the heat flow, the usual heat semigroup approach showed that the following function is increasing in time $t$:
\[
Q(t) := \int_{E_0}  \prod_{j = 1}^m (e^{t\Delta_{E^j}}g_j)^{d_j}( \Q_j x) dx, \quad t \ge 0.
\]
In fact, our Theorem \ref{mthm} implies $Q'(t) \ge 0$: Fix a time $t \ge 0$ and a positive number $s$. Define the function $f_{1,t}$ on $E_1 = E_0$ by
\[
f_{1,t}(x) := \prod_{j = 1}^m (e^{t\Delta_{E^j}}g_j)^{d_j}( \Q_j x), \quad x \in E_1 = E_0.
\]
It follows from Theorem \ref{mthm} that
\begin{align}\label{relheat}
(e^{s \Delta_{E_1}}f_{1,t})(x) \le \prod_{j = 1}^m (e^{(t + s)\Delta_{E^j}}g_j)^{d_j}( \Q_j x), \quad \forall \, x \in E_1 = E_0.
\end{align}
Taking integral in \eqref{relheat}, since the heat semigroup preserves the $L^1$ norm, we obtain
\[
\int_{E_0} f_{1,t}(x) dx \le \int_{E_0} \prod_{j = 1}^m (e^{(t + s)\Delta_{E^j}}g_j)^{d_j}( \Q_j x) dx,
\]
which is exactly $Q(t) \le Q(t+s)$. Note that in \cite[\S~2.5]{BB19}, the authors proved that the function $f_{1,t}$ satisfies $\partial_t f_{1,t} \ge \Delta_{E_1}f_{1,t}$, which also implies \eqref{relheat} by the maximum principle for heat equations (the argument is quite similar to the one in the proof of Theorem \ref{mthm} so we omit the details).
\end{remark}

The proof of Theorem \ref{mthm} relies on the maximum principle for parabolic equations inspired by \cite{BC04, BH09}. See Lemma \ref{lem0} below for more details. We remark that in \cite{ILS24}, similar idea also works for the proof of Borell--Brascamp--Lieb inequality, which is another generalization of Pr\'ekopa--Leindler inequality.

However, to use such maximum principle, we need some regularity of the initial data. Since we should take the relation \eqref{ass3} into consideration when we do the approximation, it turns out that, instead of continuity, the semicontinuity is the suitable regularity for our purpose. And general $L^1$ functions can be approximated by upper/lower semicontinuous functions from below/above, which is guaranteed by the famous Vitali--Carath\'eodory Theorem. See Theorem \ref{TVC} below for more details.

As a corollary, from the long-time behavior of the heat equation (see Lemma \ref{lth} below), we obtain the Forward-Reverse Brascamp--Lieb constant in this geometric setting.

\begin{corollary}\label{cor1}
If $(\b{c}, \b{d}, \Q)$ is a geometric Forward-Reverse Brascamp--Lieb datum, then $C(\b{c}, \b{d}, \Q) = 1$.
\end{corollary}

\medskip

Furthermore, we can show that the heat semigroup approach only works for Forward-Reverse Brascamp--Lieb datum which is essentially geometric. In other words, the heat flow preservation also gives a characterization of data which are equivalent to geometric ones. To be more precise, the following theorem holds.

\begin{theorem}\label{mthm2}
Let $(\b{c}, \b{d}, \Q)$ be a Forward-Reverse Brascamp--Lieb datum.   Then the following are equivalent:
\begin{enumerate}[(i)]
  \item The datum $(\b{c}, \b{d}, \Q)$ is equivalent to a geometric one; 

  \item There exist $A_i \in \SS(E_i), 1 \le i \le k$ and $A^j \in \SS(E^j), 1 \le j \le m$ such that if all the functions $(f_i)_{1 \le i \le k}, (g_j)_{1 \le j \le m}$ are non-negative and integrable, and \eqref{ass3} holds, namely
\[
\prod_{i = 1}^k f_i^{c_i}(\pi_{E_i} x) \le \prod_{j = 1}^m g_j^{d_j}(\pi_{E^j} \Q x), \qquad \forall \, x \in E_0,
\]
then we have
\begin{align}\label{heat3}
\prod_{i = 1}^k (e^{t\Delta_{E_i, A_i}}f_i)^{c_i}(\pi_{E_i} x) \le \prod_{j = 1}^m (e^{t\Delta_{E^j, A^j}}g_j)^{d_j}(\pi_{E^j} \Q x), \qquad \forall \, x \in E_0, t > 0;
\end{align}

  \item There exist Gaussian extremizers.
\end{enumerate}
\end{theorem}

The proofs of Theorem \ref{mthm} and  Theorem \ref{mthm2} will be provided in Section \ref{s2} and Section \ref{s3} respectively.

\section{Proof of Theorem \ref{mthm}}\label{s2}
\setcounter{equation}{0}

In this section we give the proof of our main theorem, Theorem \ref{mthm}. However, we split the proof into four parts. In Subsection \ref{ss21} we prove the main ingredient of the proof: the maximum principle for parabolic equations. Then in Subsection \ref{ss22}, we establish three lemmas to deal with  corresponding three assumptions of Lemma \ref{lem0}. The proof under an extra  assumption on the regularity of initial data is provided in Subsection \ref{ss23}. Finally in Subsection \ref{ss24} we end the proof with an approximation process by using the Vitali--Carath\'eodory Theorem.

\subsection{Maximum principle for parabolic equations}\label{ss21}

In this section we prove a slightly improved version of \cite[Lemma 1]{BH09}, which is closely related to the maximum principle for parabolic equations. To be more precise, we lower the regularity required for the initial datum at time $0$. Although the proof is similar, for the sake of completeness, we include the proof here.

\begin{lemma}\label{lem0}
Let $V$ be a Euclidean space and $\C \in C^2( V \times (0,+\infty))$. Assume furthermore that: 
\begin{enumerate}
\item for any point $(v_*, t_*) \in V \times (0,+\infty)$ such that 
\begin{align}\label{cases}
\begin{cases}
\C(v_*,t_*) \le 0 \\
\nabla_V \C (v_*, t_*) = 0 \\
\nabla^2_V \C(v_*,t_*) \ge 0
\end{cases},
\end{align}
we have $\partial_t \C (v_*, t_*) \ge 0$;\\
\item for every $T > 0$, the following estimate holds:
\begin{align}\label{asss1}
\liminf_{|v|_V \to +\infty} \left( \inf_{t \in (0, T]} \C(v,t) \right) \ge 0;
\end{align}
\item for every $v_0 \in V$, the following estimate holds:
\begin{align}\label{asss2}
\liminf_{v \to v_0, t \to 0^+} \C(v,t) \ge 0.
\end{align}
Here all the ``$\liminf$''s can take the value $+\infty$. 
\end{enumerate}
Then we obtain
\[
\C(v,t) \ge 0, \qquad \forall \, (v,t) \in V \times (0,+\infty).
\]
\end{lemma}

\begin{proof}
Assume on the contrary that there is a point $(v_0,t_0) \in V \times (0,+\infty)$ such that 
\[
\C(v_0,t_0) = - 2\delta < 0.
\]
We can pick an $\epsilon > 0$ such that $\C(v_0, t_0) + \epsilon t_0 < -\delta < 0 $. Also pick a $T \in [t_0, +\infty)$. Now we consider the following function on $V \times (0,T]$:
\[
\C_\epsilon(v,t) := \C(v,t) + \epsilon t.
\]
From the construction above we know that 
\[
\inf_{(v,t) \in V \times (0,T]} \C_\epsilon(v,t) \le \C_\epsilon(v_0, t_0) < -\delta < 0.
\]
By definition, there is a sequence $\{(v_n,t_n)\}_{n \ge 1} \subset V \times (0,T]$ such that 
\[
\C_\epsilon(v_n, t_n) \to \inf_{(v,t) \in V \times (0,T]} \C_\epsilon(v,t), \qquad \mbox{as $n \to +\infty$}.
\]
It follows from \eqref{asss1} that there exists a $R > 0$ such that if $|v|_V > R$ we have
\[
\inf_{t \in (0, T]} \C_\epsilon(v,t) \ge \inf_{t \in (0, T]} \C(v,t) \ge -\delta,
\]
which implies for $n$ large enough, we must have $|v_n|_V \le R$. After passing to some subsequence, we can assume that $(v_n,t_n) \to (v_*,t_*) \in V \times [0,T]$. We claim that $t_* > 0$, otherwise we have
\[
\liminf_{v \to v_*, t \to 0^+} \C(v,t) \le  \liminf_{n \to +\infty} \C(v_n, t_n) \le \lim_{n \to +\infty} \C_\epsilon(v_n, t_n) = \inf_{(v,t) \in V \times (0,T]} \C_\epsilon(v,t) < 0,
\]
which contradicts \eqref{asss2}. Now by continuity of $\C$ we have \[\C(v_*,t_*) \le \C_\epsilon(v_*,t_*) =  \inf_{(v,t) \in V \times (0,T]} \C_\epsilon(v,t) < 0.\]
This implies that at the point $(v_*,t_*)$, we have \eqref{cases} holds and $\partial_t \C(v_*,t_*) + \epsilon \le 0$. But \eqref{cases} implies $\partial_t \C(v_*,t_*) \ge 0$ by the first assumption and this leads to a contradiction. 
\end{proof}

\subsection{Some preparation lemmas}\label{ss22}

Before giving the proof of Theorem \ref{mthm}, we make some preparations to deal with the three assumptions in Lemma \ref{lem0}.

\begin{lemma}\label{lem1}
Assume $(\b{c}, \b{d}, \Q)$ is a geometric Forward-Reverse Brascamp--Lieb datum. 
\begin{enumerate}
\item We have 
\begin{align}\label{key1}
\sum_{i = 1}^k\frac{1}{c_i} \left|\sum_{j = 1}^m d_j \Q_{ji}^* v^j\right|_{E_i}^2 \le \sum_{j = 1}^m d_j |v^j|_{E^j}^2, \qquad \forall \, v^j \in E^j, 1 \le j \le m.
\end{align}
\item If  $X_i \in \SSS(E_i), 1 \le i \le k$ and $Y_j \in \SSS(E^j), 1 \le j \le m$ satisfying
\begin{align}\label{key2}
\sum_{i = 1}^k c_i \pi_{E_i}^* X_i \pi_{E_i} \le \Q^* \left( \sum_{j = 1}^m d_j \pi_{E^j}^* Y_j \pi_{E^j} \right) \Q,
\end{align}
then it holds that 
\begin{align}\label{key3}
\sum_{i = 1}^k c_i \tr(X_i) \le \sum_{j = 1}^m d_j \tr(Y_j).
\end{align}
\end{enumerate}
\end{lemma}

\begin{proof}
With the identification \ref{iden},
if we write \eqref{key1} in the language of matrix, to prove \eqref{key1} is equivalent to proving  the following:
\[
\Lambda_{\b{d}} \Q \Lambda_{\b{c}}^{-1} \Q^* \Lambda_{\b{d}} \le \Lambda_{\b{d}},
\]
which is equivalent to $A^*A \le \I_{E^0}$, where $A := \Lambda_{\b{c}}^{-\frac{1}{2}} \Q^*\Lambda_{\b{d}}^{\frac{1}{2}}$. On the other hand, from \eqref{ass2} we have $A A^* \le \I_{E_0}$, which is equivalent to saying that $\|A^*\|_{E_0 \to E^0} \le 1$, and thus we have $\|A\|_{E^0 \to E_0} \le 1$, which implies $A^*A \le \I_{E^0}$. This ends the proof of  the first part. For the second part, in the same spirit we write \eqref{key2} as
\[
\begin{pmatrix}
c_1 X_1 & 0 & \ldots & 0 \\
0 & c_2 X_2 &  \ldots & 0\\
\vdots & \vdots & \ddots & \vdots \\
0 & 0 & \ldots & c_k X_k
\end{pmatrix} \le \Q^* \begin{pmatrix}
d_1 Y_1 & 0 & \ldots & 0  \\
0 & d_2 Y_2 &  \ldots & 0 \\
\vdots & \vdots & \ddots & \vdots \\
0 & 0 & \ldots & d_m Y_m
\end{pmatrix} \Q
\]
 This implies
\[
\sqrt{\Sigma}\begin{pmatrix}
c_1 X_1 & 0 & \ldots & 0 \\
0 & c_2 X_2 &  \ldots & 0\\
\vdots & \vdots & \ddots & \vdots \\
0 & 0 & \ldots & c_k X_k
\end{pmatrix} \sqrt{\Sigma} \le \sqrt{\Sigma} \Q^* \begin{pmatrix}
d_1 Y_1 & 0 & \ldots & 0  \\
0 & d_2 Y_2 &  \ldots & 0 \\
\vdots & \vdots & \ddots & \vdots \\
0 & 0 & \ldots & d_m Y_m
\end{pmatrix} \Q \sqrt{\Sigma},
\]
where $\Sigma$ is the matrix appearing in the definition of geometric Forward-Reverse Brascamp--Lieb datum. Taking trace on both sides and using the property that $\tr(\mathbf{A}\mathbf{B}) = \tr(\mathbf{B}\mathbf{A})$ for suitable linear maps $\mathbf{A}, \mathbf{B}$, we obtain 
\[
\tr\left( \Sigma\begin{pmatrix}
c_1 X_1 & 0 & \ldots & 0 \\
0 & c_2 X_2 &  \ldots & 0\\
\vdots & \vdots & \ddots & \vdots \\
0 & 0 & \ldots & c_k X_k
\end{pmatrix}  \right) \le \tr\left( \Q \Sigma \Q^* \begin{pmatrix}
d_1 Y_1 & 0 & \ldots & 0  \\
0 & d_2 Y_2 &  \ldots & 0 \\
\vdots & \vdots & \ddots & \vdots \\
0 & 0 & \ldots & d_m Y_m
\end{pmatrix} \right).
\]
Since $(\b{c}, \b{d}, \Q)$ is geometric, the matrices have the following block forms:
\begin{align*}
\Sigma\begin{pmatrix}
c_1 X_1 & 0 & \ldots & 0 \\
0 & c_2 X_2 &  \ldots & 0\\
\vdots & \vdots & \ddots & \vdots \\
0 & 0 & \ldots & c_k X_k
\end{pmatrix} &= \begin{pmatrix}
c_1 X_1 & * & \ldots & * \\
* & c_2 X_2 &  \ldots & *\\
\vdots & \vdots & \ddots & \vdots \\
* & * & \ldots & c_k X_k
\end{pmatrix} \\
\Q \Sigma \Q^* \begin{pmatrix}
d_1 Y_1 & 0 & \ldots & 0  \\
0 & d_2 Y_2 &  \ldots & 0 \\
\vdots & \vdots & \ddots & \vdots \\
0 & 0 & \ldots & d_m Y_m
\end{pmatrix} &= \begin{pmatrix}
d_1 Y_1 & * & \ldots & *  \\
* & d_2 Y_2 &  \ldots & * \\
\vdots & \vdots & \ddots & \vdots \\
* & * & \ldots & d_m Y_m
\end{pmatrix}.
\end{align*}
So the trace inequality yields \eqref{key3}. Thus we prove the lemma.
\end{proof}

\begin{lemma}\label{lem2}
Assume $f$ is a bounded function with compact support on some Euclidean space $V$. Then  we have
\[
\lim_{|x|_V \to +\infty}\sup_{t \in (0,+\infty)}|(e^{t\Delta_V} f)(x) | = 0.
\]
\end{lemma}

\begin{proof}
Assume $|f| \le M$ and the support of $f$ is contained in $K := \{v: \, |v|_V \le R \}$. Then combining \eqref{conv} with \eqref{exphk}, we obtain
\[
|(e^{t\Delta_V} f)(x) | \le \int_V |f(y)| p_t^{V}(x - y) dy \le \frac{M}{(4\pi t)^{\frac{\mathrm{dim} V }{2}}} \int_K     e^{-\frac{|x- y|^2_V}{4t}} dy.
\]
Note that when $|x|_V \ge 2R$ and $y \in K$, we have $|x - y|_V \ge |x|_V - |y|_V \ge |x|_V - R \ge \frac{|x|_V}{2}$. Thus the last term can be bounded by
\[
\frac{C(M,R,\mathrm{dim} V )}{t^{\frac{\mathrm{dim} V }{2}}} e^{-\frac{|x|_V^2}{16t}} \le \frac{C'(M,R,\mathrm{dim} V )}{|x|_V^{\mathrm{dim} V}} \to 0
\]
as $|x|_V \to +\infty$. Since the last term is independent of $t$, we can take supremum w.r.t. $t > 0$.
\end{proof}

\begin{lemma}\label{lem3}
Assume $f$ and $g$ are integrable functions on some Euclidean space $V$.  
\begin{enumerate}
\item If $f$ is upper semicontinuous and bounded from above, then for every $x_0 \in V$, we have
\[
\limsup_{x \to x_0, t \to 0^+} (e^{t\Delta_V} f)(x) \le f(x_0).
\] 
\item If $g$ is lower semicontinuous and bounded from below, then for every $x_0 \in V$, we have
\[
\liminf_{x \to x_0, t \to 0^+} (e^{t\Delta_V} g)(x) \ge g(x_0).
\] 
\end{enumerate}
\end{lemma}

\begin{proof}
It suffices to prove the first assertion since the proof for the second one is quite similar. In this proof we use $B(x,r)$ to denote the ball in $V$ centred at $x$ with radius $r$. Now fix a point $x_0 \in V$. For every $\epsilon > 0$, by the upper semicontinuity of $f$ at $x_0$, there exists a $\delta > 0$ such that 
\[
f(z) \le f(x_0) + \frac{\epsilon}{2}, \qquad \forall \, z \in B(x_0,2 \delta).
\]
Then from \eqref{conv}, we have
\begin{align}\label{Delta1}
(e^{t\Delta_V} f)(x) = \int_V f(x - y) p_t^{V}(y) dy = \int_{B(0,\delta)} f(x - y) p_t^{V}(y) dy 
+  \int_{B(0,\delta)^c} f(x - y) p_t^{V}(y) dy.
\end{align}
For the first term, if $x \in B(x_0, \delta)$ and $y \in B(0,\delta)$, we have $x - y \in B(x_0,2 \delta)$ and thus 
\begin{align}\label{Delta2}
\int_{B(0,\delta)} f(x - y) p_t^{V}(y) dy \le \left( f(x_0) + \frac{\epsilon}{2} \right) \int_{B(0,\delta)} p_t^{V}(y) dy = f(x_0) + \frac{\epsilon}{2} - \left( f(x_0) + \frac{\epsilon}{2} \right) \int_{B(0,\delta)^c} p_t^{V}(y) dy.
\end{align}
Now for the second terms in \eqref{Delta1} and \eqref{Delta2}, by \eqref{exphk} and a change of variables, we have
\begin{align*}
\left| \left( f(x_0) + \frac{\epsilon}{2} \right) \int_{B(0,\delta)^c} p_t^{V}(y) dy \right|  &\le  \frac{\left| f(x_0) + \frac{\epsilon}{2} \right|}{(4\pi t)^{\frac{\mathrm{dim} V }{2}}} \int_{B(0,\delta)^c}  e^{-\frac{| y|^2_V}{4t}} dy = \left| f(x_0) + \frac{\epsilon}{2} \right|  \int_{B(0,\delta/\sqrt{4\pi t})^c}  e^{-\pi| z|^2_V} dz,\\
\int_{B(0,\delta)^c} f(x - y) p_t^{V}(y) dy &\le \frac{M}{(4\pi t)^{\frac{\mathrm{dim} V }{2}}} \int_{B(0,\delta)^c}  e^{-\frac{| y|^2_V}{4t}} dy =  M \int_{B(0,\delta/\sqrt{4\pi t})^c}  e^{-\pi| z|^2_V} dz,
\end{align*}
where $M > 0$ is an upper bound of $f$. The last terms of the two estimates above can be made sufficient small as $t \to 0^+$ by the dominated convergence theorem. So there exists a $\eta > 0$ such that when $t \in (0, \eta)$,
\[
- \left( f(x_0) + \frac{\epsilon}{2} \right) \int_{B(0,\delta)^c} p_t^{V}(y) dy + \int_{B(0,\delta)^c} f(x - y) p_t^{V}(y) dy \le \frac{\epsilon}{2}.
\]
Combining the estimates of all the terms together, we conclude that for every $\epsilon > 0$ given, there exist $\delta > 0$ and $\eta > 0$ such that when $x \in B(x_0, \delta)$ and $t \in (0, \eta)$, we obtain
\[
(e^{t\Delta_V} f)(x) \le f(x_0) + \epsilon.
\]
\end{proof}

\subsection{The proof under suitable regularity}\label{ss23}

The target of this subsection is to prove the following proposition, which is a special case of Theorem \ref{mthm}.

\begin{proposition}\label{p1}
Let $(\b{c}, \b{d}, \Q)$ be a geometric Forward-Reverse Brascamp--Lieb datum. Assume that functions $(f_i)_{1 \le i \le k}$ are non-negative,  upper semicontinuous, with compact support, and functions $(g_j)_{1 \le j \le m}$ are non-negative, lower semicontinuous. Then if \eqref{ass3} holds, namely
\[
\prod_{i = 1}^k f_i^{c_i}(\pi_{E_i} x) \le \prod_{j = 1}^m g_j^{d_j}(\pi_{E^j} \Q x), \qquad \forall \, x \in E_0,
\]
we have
\begin{align}\label{heat2}
\prod_{i = 1}^k (e^{t\Delta_{E_i}}f_i)^{c_i}(\pi_{E_i} x) \le \prod_{j = 1}^m (e^{t\Delta_{E^j}}g_j)^{d_j}(\pi_{E^j} \Q x), \qquad \forall \, x \in E_0, t > 0.
\end{align}
\end{proposition}

\begin{proof}
For any positive number $\delta > 0$, we define $u_i := \log(e^{t\Delta_{E_i}}f_i), 1 \le i \le k, g_{j}^\delta := g_j + \delta, v_j^\delta := \log(e^{t\Delta_{E^j}}g_j^\delta), 1 \le j \le m$.
Since $e^{t\Delta_{E^j}}g_j^\delta = e^{t\Delta_{E^j}}g_j + \delta$, by definition we know that all the functions are smooth and
\begin{align}\label{ing1}
\partial_t u_i &=  \Delta_{E_i} u_i + |\nabla_{E_i} u_i|_{E_i}^2, \quad \forall \, 1 \le i \le k,\\
\label{ing2}
\partial_t v_j^\delta &=  \Delta_{E^j} v_j^\delta + |\nabla_{E^j} v_j^\delta|_{E^j}^2, \quad \forall \, 1 \le j \le m. 
\end{align}

Now on the space $E_0 \times (0,+\infty)$ we construct
\[
\C(x,t) := \sum_{j = 1}^m d_j v_j^\delta(\pi_{E^j} \Q x, t) - \sum_{i = 1}^k c_i u_i(\pi_{E_i} x,t).
\]

We now show that this $\C$ satisfies all the three assumptions of Lemma \ref{lem0}. Assume at the point $(x_*,t_*)$ we have $\nabla_{E_0} \C (x_*,t_*) = 0$ and $\nabla^2_{E_0} \C (x_*,t_*) \ge 0$. From $\nabla_{E_0} \C (x_*,t_*) = 0$ we obtain that 
\begin{align}\label{ing5}
 c_i\nabla_{E_i} u_i (\pi_{E_i} x_*,t_*) =  \sum_{j = 1}^m d_j \Q_{ji}^* \nabla_{E^j} v_j^\delta(\pi_{E^j} \Q x_*, t_*), \quad \forall \, 1 \le i \le k.
\end{align}
Using \eqref{key1} with $v^j = \nabla_{E^j} v_j^\delta(\pi_{E^j} \Q x_*, t_*), 1 \le j \le m$, we have 
\begin{align}\label{ing4}
\sum_{i = 1}^k  c_i |\nabla_{E_i} u_i (\pi_{E_i} x_*,t_*) |_{E_i}^2 \le  \sum_{j = 1}^m d_j  | \nabla_{E^j} v_j^\delta(\pi_{E^j} \Q x_*, t_*)|_{E^j}^2.
\end{align}

From $\nabla^2_{E_0} \C (x_*,t_*) \ge 0$ we know that \eqref{key2} holds with  $X_i = \nabla^2_{E_i} u_i  (\pi_{E_i} x_*,t_*),  1 \le i \le k$, and $Y_j = \nabla^2_{E^j} v_j^\delta(\pi_{E^j} \Q x_*, t_*), 1 \le j \le m$. It follows from \eqref{key3} that
\begin{align}\label{ing3}
\sum_{i = 1}^k c_i \Delta_{E_i} u_i (\pi_{E_i} x_*,t_*) \le \sum_{j = 1}^m d_j \Delta_{E^j} v_j^\delta(\pi_{E^j} \Q x_*, t_*).
\end{align}

Combining \eqref{ing1}-\eqref{ing2} with \eqref{ing4}-\eqref{ing3}, we have
\begin{align*}
&\partial_t \C(x_*,t_*) = \sum_{j = 1}^m d_j \partial_t  v_j^\delta(\pi_{E^j} \Q x_*, t_*) - \sum_{i = 1}^k c_i  \partial_t u_i(\pi_{E_i} x_*,t_*) \\
= \, & \sum_{j = 1}^m d_j (\Delta_{E^j} v_j^\delta(\pi_{E^j} \Q x_*, t_*) + |\nabla_{E^j} v_j^\delta(\pi_{E^j} \Q x_*, t_*)|_{E^j}^2) \\
- \, & \sum_{i = 1}^k c_i (\Delta_{E_i} u_i(\pi_{E_i} x_*,t_*) + |\nabla_{E_i} u_i (\pi_{E_i} x_*,t_*)|_{E_i}^2) \ge 0,
\end{align*}
which shows the first assumption of Lemma \ref{lem0} holds. Now Lemma \ref{lem2} implies 
\[
\liminf_{|x|_{E_0} \to +\infty}\sup_{t \in (0,+\infty)}\sum_{i = 1}^k c_i u_i(\pi_{E_i} x,t) \to -\infty.
\]
Together with the fact that $(v_j^\delta)_{1 \le j \le m}$ are bounded from below by construction, it is easy to show that the second assumption of Lemma \ref{lem0} holds. Finally for the third assumption of Lemma \ref{lem0}, by Lemma \ref{lem3}, for any $x_0 \in E_0$ we have
\[
\liminf_{x \to x_0, t \to 0^+} \C(x,t) \ge \sum_{j = 1}^m d_j \log(g_j^\delta) (\pi_{E^j} \Q x_0) - \sum_{i = 1}^k c_i \log(f_i) (\pi_{E_i} x_0) \ge 0.
\]
Thus Lemma \ref{lem0} applies to the function $\C$ here and we obtain that 
\[
\sum_{j = 1}^m d_j v_j^\delta(\pi_{E^j} \Q x, t) - \sum_{i = 1}^k c_i u_i(\pi_{E_i} x,t) \ge 0.
\]
Since $\delta > 0$ is arbitrary, letting $\delta \to 0^+$, we prove this proposition.

\end{proof}

\subsection{Approximation process}\label{ss24}

In this subsection we remove the regularity assumptions on functions $(f_i)_{1 \le i \le k}, (g_j)_{1 \le j \le m}$ in Proposition \ref{p1} by using the famous Vitali--Carath\'eodory Theorem, which states that any integrable function can be approximated from above by lower semicontinuous functions and from below by upper semicontinuous functions. We refer to \cite[\S~2.25]{R87} for the proof. 

\begin{theorem}[Vitali--Carath\'eodory Theorem]\label{TVC}
Given a Euclidean space $V$ and a function $f \in L^1(V)$, for every $\epsilon > 0$, there exist two functions $u$ and $v$, such that $u \le f \le v$, $u$ is upper semicontinuous and bounded from above, $v$ is lower semicontinuous and bounded from below, and 
\[
\int_V (v - u) \le \epsilon.
\]
Furthermore, if $f \ge 0$, we can pick $u$ in a way such that $u \ge 0$ and $u$ has compact support.
\end{theorem}

\begin{remark}
The last part of Theorem \ref{TVC} is not stated in \cite[\S~2.25]{R87} but it can be seen from the proof in \cite[\S~2.25]{R87}.
\end{remark}

\begin{proof}[Proof of Theorem \ref{mthm}]
It follows from Theorem \ref{TVC} that for each $1 \le i \le k$, there exists a sequence $\{u_{i,n}\}_{n \ge 1} \subset L^1(E_i)$ such that for each $n \ge 1$, $u_{i,n}$ is non-negative, upper semicontinuous, with compact support, $u_{i,n} \le f_i$, and $u_{i,n} \to f_i$ in $L^1(E_i)$ as $n \to +\infty$. Similarly, for each $1 \le j \le m$, there exists a sequence $\{v_{j,n}\}_{n \ge 1} \subset L^1(E^j)$ such that for each $n \ge 1$, $v_{j,n}$ is non-negative, lower semicontinuous, $g_j \le v_{j,n}$, and $v_{j,n} \to g_j$ in $L^1(E^j)$  as $n \to +\infty$. Now, for each $n \ge 1$, by \eqref{ass3} we have 
\[
\prod_{i = 1}^k (u_{i,n})^{c_i}(\pi_{E_i} x) \le \prod_{j = 1}^m (v_{j,n})^{d_j}(\pi_{E^j} \Q x), \qquad \forall \, x \in E_0.
\]
Applying Proposition \ref{p1}, we obtain 
\begin{align*}
\prod_{i = 1}^k (e^{t\Delta_{E_i}}u_{i,n})^{c_i}(\pi_{E_i} x) \le \prod_{j = 1}^m (e^{t\Delta_{E^j}}v_{j,n})^{d_j}(\pi_{E^j} \Q x), \qquad \forall \, x \in E_0, t > 0.
\end{align*}
Now fix $x \in E_0$ and $t > 0$. We let $n \to +\infty$. Since all the heat kernels are bounded functions for fixed time, the $L^1$ convergences imply the pointwise convergences, and we prove Theorem \ref{mthm}.
\end{proof}

\begin{remark}
The main difficulty of the approximation process lies in the preservation of the relation \eqref{ass3}. For example, it is easy to assume that each $f_i$ is bounded and has compact support, but the boundedness for each $g_j$ is harder since we need to take care of the relation  \eqref{ass3}. To be more precise, even a change of values on a small set in $E_0$ will affect the values of $(g_j)_{1 \le j \le m}$ a lot. For example, if $\Q$ is a rotation on $\R^2$. Let $L := [-\epsilon,\epsilon]\times[-M,M]$. Then with $\epsilon$ small the measure of the projection of $L$ onto the first variable is small. However, we cannot control the measure of the projection of $\Q(L)$ onto the first variable for every fixed $\Q$. This explains why we used  upper/lower semicontinuous functions instead of  continuous functions.
\end{remark}

To prove Corollary \ref{cor1} we need the following long-time behavior of the heat equation, which is a direct result of \eqref{conv}, \eqref{exphk}, and dominated convergence theorem. 
\begin{lemma}\label{lth}
Assume $V$ is a Euclidean space, $A \in \SS(V)$, and $f \in L^1(V)$. Then for every $x \in V$, we have
\[
\lim_{t \to +\infty}(4\pi t)^{\frac{\mathrm{dim} V}{2}}(e^{t \Delta_{V, A}} f)(2\sqrt{t} x)= \frac{1}{\sqrt{\det(A)}} \, e^{-\l A^{-1} x,x \r_V} \left(\int_V f \right).
\]
\end{lemma}

\begin{proof}[Proof of Corollary \ref{cor1}]
Without loss of generality we can assume all  functions $(f_i)_{1 \le i \le k}, (g_j)_{1 \le j \le m}$ are  integrable. Then Theorem \ref{mthm} yields 
\begin{align*}
\prod_{i = 1}^k (e^{t\Delta_{E_i}}f_i)^{c_i}(0) \le \prod_{j = 1}^m (e^{t\Delta_{E^j}}g_j)^{d_j}(0), \qquad \forall \, t > 0.
\end{align*}
Multiplying both sides with
\[
\prod_{i = 1}^k (4\pi t)^{\frac{c_i \mathrm{dim} (E_i)}{2}} = \prod_{j = 1}^m (4\pi t)^{\frac{d_j \mathrm{dim} (E^j)}{2}},
\]
where the equality is given by \eqref{ass0}, and letting $t \to +\infty$, we obtain
\[
\prod_{i = 1}^k \left( \int_{E_i} f_i \right)^{c_i} \le \prod_{j = 1}^m  \left( \int_{E^j} g_j \right)^{d_j}
\]
by Lemma \ref{lth}. This implies $ C(\b{c}, \b{d}, \Q) \le 1$. To prove the opposite direction, it suffices to notice that from \eqref{ass2} we have
\[
\sum_{j = 1}^m  d_j |\pi_{E^j} \Q x|_{E^j}^2 \le \sum_{i = 1}^k c_i |\pi_{E_i} x|^2_{E_i}, \qquad \forall \, x \in E_0,
\]
or equivalently
\[
\prod_{i = 1}^k \left( e^{- |\pi_{E_i} x|^2_{E_i}} \right)^{c_i} \le \prod_{j = 1}^m \left( e^{-|\pi_{E^j} \Q x|_{E^j}^2} \right)^{d_j}, \qquad \forall \, x \in E_0.
\]
As a result, the choice $f_i = e^{-  |\cdot|^2_{E_i} }, 1 \le i \le k$ and $g_j = e^{-  |\cdot|_{E^j}^2}, 1 \le j \le m$ gives $ C(\b{c}, \b{d}, \Q) \ge 1$.
\end{proof}

\section{Proof of Theorem \ref{mthm2}}\label{s3}
\setcounter{equation}{0}

\begin{proof}[Proof of Theorem \ref{mthm2}]

(i) $\Rightarrow$ (ii): From (i) there exists a geometric Forward-Reverse Brascamp--Lieb datum $(\b{c}, \b{d}, \Q')$  such that
\[
\Q_{ji}' = D_j^{-1} \Q_{ji} C_i^{-1}, \qquad \forall \, 1 \le i \le k, 1 \le j \le m
\]
for some invertible linear maps $C_i : E_i \to E_i, 1 \le i \le k$, $D_j : E^j \to E^j, 1 \le j \le m$. Now define 
\[
\C := \sum_{i =1}^k  \pi_{E_i}^*  C_i  \pi_{E_i} \qquad \D := \sum_{j =1}^m \pi_{E^j}^* D_j \pi_{E^j}.
\]
With these notations we have $\Q = \D \Q' \C$. In the following we use $\circ$ to denote the composition of maps. Since \eqref{ass2} holds, we have
\[
\prod_{i = 1}^k F_i^{c_i}(\pi_{E_i} x) \le \prod_{j = 1}^m G_j^{d_j}(\pi_{E^j} \Q'  x), \qquad \forall \, x \in E_0,
\]
where $F_i := f_i \circ C_i^{-1}, 1 \le i \le k$, $G_j := g_j \circ D_j, 1 \le j \le m$. It follows from Theorem \ref{mthm} that
\[
\prod_{i = 1}^k (e^{t\Delta_{E_i}}F_i)^{c_i}(\pi_{E_i} x) \le \prod_{j = 1}^m (e^{t\Delta_{E^j}}G_j)^{d_j}(\pi_{E^j} \Q' x), \qquad \forall \, x \in E_0, t > 0.
\] 
By a direct calculation using \eqref{conv} and \eqref{exphk}, we obtain
\[
\prod_{i = 1}^k (e^{t\Delta_{E_i,A_i}}f_i)^{c_i}(C_i^{-1} \pi_{E_i} x) \le \prod_{j = 1}^m (e^{t\Delta_{E^j,A^j}}g_j)^{d_j}( D_j \pi_{E^j} \Q' x), \qquad \forall \, x \in E_0, t > 0.
\] 
where $A_i := C_i^{-1} (C_i^{-1})^*, 1 \le i \le k$, $A^j := D_j D_j^*, 1 \le j  \le m$. This implies \eqref{heat3}.

(ii) $\Rightarrow$ (iii): Assume now \eqref{heat3} holds. We replace $x$ with $2\sqrt{t} x$, multiply both sides with
\[
\prod_{i = 1}^k (4\pi t)^{\frac{c_i \mathrm{dim} (E_i)}{2}} = \prod_{j = 1}^m (4\pi t)^{\frac{d_j \mathrm{dim} (E^j)}{2}},
\]
as before, and let $t \to +\infty$. It follows from Lemma \ref{lth} that
\begin{align}\label{add5}
\prod_{i = 1}^k \left( \int_{E_i} f_i \right)^{c_i} e^{- c_i \langle A_i^{-1} \pi_{E_i}x,\pi_{E_i}x \rangle_{E_i} } \le \frac{\prod_{i = 1}^k (\det (A_i))^{\frac{c_i}{2}}}{\prod_{j = 1}^m (\det (A^j))^{\frac{d_j}{2}}} \prod_{j = 1}^m  \left( \int_{E^j} g_j \right)^{d_j} e^{- d_j \langle (A^j)^{-1} \pi_{E^j} \Q x, \pi_{E^j} \Q x \rangle_{E^j} }
\end{align}
holds for all $x \in E_0$. Now let $x = 0$, we obtain
\[
C(\b{c},\b{d},\Q) \le \frac{\prod_{i = 1}^k (\det (A_i))^{\frac{c_i}{2}}}{\prod_{j = 1}^m (\det (A^j))^{\frac{d_j}{2}}}.
\]
Now  replacing $x$ with $rx$ ($r > 0$) in \eqref{add5},  taking the logarithm, and then dividing both sides by $r^2$, we obtain
\[
- \sum_{i = 1}^k c_i \langle A_i^{-1} \pi_{E_i}x,\pi_{E_i}x \rangle_{E_i} \le \frac{C'}{r^2} - \sum_{j = 1}^m d_j \langle (A^j)^{-1} \pi_{E^j} \Q x, \pi_{E^j} \Q x \rangle_{E^j}.
\]
Here $C'$ is a constant independent of $x \in E_0$ and $r > 0$. Letting $r \to +\infty$ we have 
\[
\sum_{j = 1}^m d_j \langle (A^j)^{-1} \pi_{E^j} \Q x, \pi_{E^j} \Q x \rangle_{E^j} \le  \sum_{i = 1}^k c_i \langle A_i^{-1} \pi_{E_i}x,\pi_{E_i}x \rangle_{E_i},
\]
or equivalently
\[
\prod_{i = 1}^k \left( e^{-  \langle A_i^{-1} \pi_{E_i}x,\pi_{E_i}x \rangle_{E_i} } \right)^{c_i} \le \prod_{j = 1}^m \left( e^{- \langle (A^j)^{-1} \pi_{E^j} \Q x, \pi_{E^j} \Q x \rangle_{E^j} } \right)^{d_j}.
\]
This means the choice $f_i = e^{- \langle A_i^{-1} \pi_{E_i}x,\pi_{E_i}x \rangle_{E_i} }, 1 \le i \le k$, $g_j = e^{- \langle (A^j)^{-1} \pi_{E^j} \Q x, \pi_{E^j} \Q x \rangle_{E^j} }, 1 \le j \le m$ gives 
\[
C(\b{c},\b{d},\Q) = \frac{\prod_{i = 1}^k (\det (A_i))^{\frac{c_i}{2}}}{\prod_{j = 1}^m (\det (A^j))^{\frac{d_j}{2}}} = \frac{\prod_{i = 1}^k \left( \int_{E_i} f_i \right)^{c_i}}{\prod_{j = 1}^m  \left( \int_{E^j} g_j \right)^{d_j}}.
\]
This is exactly (iii).

(iii) $\Rightarrow$ (i): Since we have Gaussian extremizers, the Forward-Reverse Brascamp--Lieb datum  $(\b{c}, \b{d}, \Q)$ is equivalent to some geometric one by \cite[Theorem 1.23]{CL21}.

\end{proof}

\section*{Acknowledgement}
\setcounter{equation}{0}

This project has received funding from (i)~the European Research Council (ERC) under the European Union's Horizon 2020 research and innovation programme (grant agreement GEOSUB, No. 945655); (ii)~the JSPS Grant-in-Aid for Early-Career Scientists (No. 24K16928).

YZ would like to thank Neal Bez for his kind invitation to Saitama University and thank Anthony Gauvan and Hiroshi Tsuji for the fruitful discussions during the visit. YZ also would like to thank the anonymous referee for many valuable suggestions and remarks.


\mbox{}\\
Ye Zhang\\
Scuola Internazionale Superiore di Studi Avanzati (SISSA)\\
Via Bonomea 265, 34136 Trieste, Italy \\
E-mail: yezhang@sissa.it  \quad or \quad  zhangye0217@gmail.com \mbox{}\\

\end{document}